\documentclass[a4paper, 11pt, reqno]{amsart}

\usepackage[usenames,dvipsnames]{color}
\usepackage{amssymb, amsmath, latexsym, graphics, graphicx, pstricks}

\newcommand{\ol}[1]{{\overline{#1}}}
\newcommand\GreenL{\mathcal{L}}
\newcommand\GreenR{\mathcal{R}}
\newcommand\GreenH{\mathcal{H}}
\newcommand\GreenD{\mathcal{D}}
\newcommand\GreenJ{\mathcal{J}}

\newtheorem{theorem}{Theorem}[section]

\newtheorem{proposition}[theorem]{Proposition}

\newtheorem{question}[theorem]{Question}
\newtheorem{corollary}[theorem]{Corollary}

\newtheorem{example}[theorem]{Example}

\begin{document}
\title{A large class of sofic monoids}


\maketitle

\begin{center}

MARK KAMBITES\footnote{Email \texttt{Mark.Kambites@manchester.ac.uk}.}

    \medskip

    School of Mathematics, \ University of Manchester, \\
    Manchester M13 9PL, \ England.

 \date{\today}

\end{center}

\begin{abstract}
We prove that a monoid is sofic, in the sense recently introduced by Ceccherini-Silberstein
and Coornaert, whenever the $\GreenJ$-class of the
identity is a sofic group, and the quotients of this group by orbit
stabilisers in the rest of the monoid are amenable. In particular, this
shows that the following are all sofic: cancellative monoids with amenable group of
units; monoids with sofic group of units and finitely many non-units; and
monoids with amenable Sch\"utzenberger groups and finitely many $\GreenL$-classes in each
$\GreenD$-class. This provides a very wide range of sofic monoids, subsuming
most known examples (with the notable exception of locally residually finite
monoids). We conclude by discussing some aspects of the definition, and posing
some questions for future research.
\end{abstract}

\section{Introduction}

\textit{Sofic groups} are a class of groups forming a common generalisation
of amenable groups and residually finite groups. They were introduced by
Gromov \cite{Gromov99}, applied to dynamical systems
by Weiss \cite{Weiss00} and subsequently studied by many authors. For a
detailed introduction, see for example \cite[Chapter 7]{Cecc10}.

Ceccherini-Silberstein and Coornaert \cite{Cecc13} have recently introduced
a definition of a \textit{sofic monoid} (see Section~\ref{sec_soficmonoid} below).
Examples of sofic monoids include sofic groups, finite (or more generally,
locally residually finite) monoids, commutative monoids, free monoids, and
cancellative left or right amenable monoids. The class of sofic monoids is closed
under operations including direct products, inductive and projective limits
and the taking of submonoids.
A notable example of a \textit{non}-sofic monoid is the bicyclic monoid;

It follows from their work and standard results in semigroup theory (see
Proposition~\ref{sec_soficproof} below for the deduction)
that the $\GreenJ$-class of the identity in a sofic monoid must be a single
group, and that this group must be sofic. Another of their results
\cite[Proposition~4.7]{Cecc13}, which
may be viewed as a partial converse, says that a monoid is sofic provided
it has no non-trivial left or right units, which in our language means
provided the $\GreenJ$-class of the identity is a trivial group.

Our main result here (Theorem~\ref{thm_main} below) can be
viewed as a very strong generalisation of this latter statement,
providing a broad class of sofic monoids which subsumes most
of the known examples listed above (with the notable exception of
locally residually finite monoids). Specifically, we
show that a monoid in which the $\GreenJ$-class of the identity is a
sofic group will be sofic, provided the quotients of this group by the
stabilisers of right translation orbits in the rest of the monoid are
amenable groups. Although the latter condition, which we term \textit{local
amenability of the action}, is rather technical in full generality, it
applies trivially in many cases of interest, such as when the group is
amenable, or the rest of the monoid is finite, or the non-unit
$\GreenR$-classes are finite. 

In Section~\ref{sec_schutz} we consider how the local amenability condition
interacts with a standard semigroup-theoretic method of using Green's
relations to structurally decompose the action of the group of units of
the monoid into two parts. We show that the local amenability condition
can be deduced from conditions on the Sch\"utzenberger groups and on the
action of the group of units on the set of $\GreenH$-classes. Consequences
include the fact that a monoid will be sofic provided the $\GreenJ$-class
of the identity is a sofic group, and each non-unit $\GreenD$-class has either
finitely many $\GreenL$-classes, or finite or abelian Sch\"utzenberger
groups.

We conclude, in Section~\ref{sec_remarks}, with some open questions
arising naturally from our results, and some discussion of aspects
of the definition of a sofic monoid.

\section{Sofic Monoids, Green's Relations and Amenability}\label{sec_soficmonoid}

In this section we briefly recall some required definitions from classical
semigroup theory \cite{Clifford61} and from the new theory of sofic monoids
\cite{Cecc13}, as well as proving a preliminary result connecting them.
We also recall the notion of \textit{amenability} for groups,
which will be required in later sections.

\subsection{Sofic Monoids} Let $M$ be a monoid with identity element $1$.
We say that $M$ is \textit{(left) sofic} if for every finite subset $K$ of $M$
and $\epsilon > 0$ there exists a finite set $X$ and a map
$$M \times X \to X, \ (m,x) \to m \cdot x$$
such that
\begin{itemize}
\item $1 \cdot x = x$ for all $x \in X$;
\item for every $g,h \in K$, the proportion of elements $x$ in $X$ such
that \\ $g \cdot (h \cdot x) = (gh) \cdot x$ is at least $1-\epsilon$; and
\item for every $g,h \in K$ with $g \neq h$, the proportion of elements
$x$ in $X$ such that $g \cdot x = h \cdot x$ is at most $\epsilon$.
\end{itemize}
We call a map satisfying these conditions a \textit{(left) $(K,\epsilon)$-action}
of $G$ on $X$.

Sofic monoids were introduced by Ceccherini-Silberstein and Coornaert
\cite{Cecc13}, with a slightly different
formal definition using a notion of a $(K,\epsilon)$-morphism to a full
transformation monoid; our formulation is trivially equivalent, being
essentially just a different notation which makes
our proofs slightly more concise. For a more leisurely introduction,
including a discussion of the basic properties of sofic monoids, see
\cite{Cecc13}.

There is an obvious dual notion of a \textit{right 
$(K,\epsilon)$-action}, which leads to a definition of a
\textit{right sofic monoid}; see \cite[Section~7]{Cecc13} for discussion this distinction. In
this paper we consider explicitly only left sofic monoids (which, for
conciseness and following
\cite{Cecc13}, we simply term ``sofic''), but of course our results have
dual statements for right sofic monois.

\subsection{Green's Relations.}
\textit{Green's relations} are five binary relations which can be defined
on any monoid $M$ as follows:
\begin{itemize}
\item $x \GreenL y$ if $Mx = My$;
\item $x \GreenR y$ if $xM = yM$;
\item $x \GreenH y$ if $Mx = My$ and $xM = yM$;
\item $x \GreenJ y$ if $MxM = MyM$; and
\item $x \GreenD y$ if there exists $z \in M$ with $x \GreenL z$ and $z \GreenR y$.
\end{itemize}
All five are equivalence relations on $M$; this is trivial in the first four
cases and requires slightly more work to show in the case of $\GreenD$. The
relations encapsulate the (left, right and two-sided) principal ideal
structure of the monoid, and play a key role in most areas of semigroup theory.
For a detailed introduction see for example \cite[Chapter 2]{Clifford61}.

\subsection{The $\GreenJ$-class of $1$ in a Sofic Monoid}

The following is a straightforward consequence of the results of
\cite{Cecc13} together with some foundational results of semigroup theory.

\begin{proposition}\label{prop_groupofunits}
Let $M$ be a sofic monoid. Then the $\GreenJ$-class of the identity is
equal to the group of units of $M$, and this is a sofic group.
\end{proposition}
\begin{proof}
Suppose for a contradiction that $j$ is in the $\GreenJ$-class of the
identity but is not a unit. Then $1 = ajb$ for some elements $a,b \in M$.
Now if $a$ is not a unit, then it is a left unit which is not a unit. On
the other hand, if $a$ is a unit then conjugating both sides by $a$ we
obtain $1 = jba$, in which case $j$ is a left unit which is not a unit.
So in all cases there are elements $x, y \in M$ such that $xy = 1$ but $x$
is not a unit.

We claim that $x$ and $y$ generate a submonoid of $M$ isomorphic to the 
bicyclic monoid. Indeed, if not, then since they satisfy the defining
relation $xy = 1$, they must generate a proper quotient of a bicyclic
monoid. But the only proper quotients of the bicyclic monoid are cyclic
groups \cite[Corollary~1.32]{Clifford61} and if they generated a group we
would also have $yx= 1$, contradicting the fact that $x$ is not a unit.

Thus, $M$ contains a copy of the bicyclic monoid. By \cite[Proposition~3.5]{Cecc13},
submonoids
of sofic monoids are sofic, but by \cite[Theorem 5.1]{Cecc13} the
bicyclic monoid is not sofic, so this gives the required contradiction.

Finally, the group of units is in particular a submonoid of $M$, so
is sofic as a monoid by \cite[Proposition~3.5]{Cecc13} and hence also as a
group by \cite[Proposition 2.4]{Cecc13}.
\end{proof}

\subsection{Amenability} For our main theorem we shall need the notion
of an \textit{amenable group}. Recall that a group is
called \textit{amenable} if
it admits a finitely additive probability measure which is invariant under
the (left or right) translation action of the group.
Most of the time we shall not make direct use of the definition of
amenability, but rather of a well-known combinatorial property
of them. A group has the \textit{F\o lner set property} if
for every finite subset $K$ of $G$ and every $\epsilon > 0$
there exists a finite subset $F$ of $G$, such that the proportion of elements
$f \in F$ satisfying $Kf \subseteq F$ is at least $1 - \epsilon$. In fact, a
group is amenable if and only if it has the F\o lner set property \cite{Folner55}.

For a full introduction to amenable groups,
we refer the reader to, for example, \cite[Chapter 4]{Cecc10}. (The
F\o lner set property has a number of slightly different but equivalent
statements; the exact formulation we use is taken from \cite[Section~2]{Weiss00} and
is not explicitly mentioned in \cite{Folner55} or \cite{Cecc10}, but it is an easy exercise
to deduce its equivalence to the properties defined in \cite[Section~4.7]{Cecc10}
and shown to be equivalent to the standard definition of amenability in
\cite[Section~4.9]{Cecc10}.)

\section{Main Theorem}\label{sec_soficproof}

Let $G$ be a group acting on a set $X$. We say that the \textit{action of 
$G$ on $X$ is locally amenable} if for every orbit in $X$, the quotient of 
$G$ by the pointwise stabiliser of the orbit is an amenable group. Note 
that, because quotients of amenable groups are amenable 
\cite[Proposition~4.5.4]{Cecc10}, every action of an amenable group is 
locally amenable. In fact it is easy to show that a group is amenable
exactly if its translation action on itself (from either side) is locally
amenable.

We are now ready to state our main theorem, which gives a very general
sufficient condition for a monoid to be sofic.

\begin{theorem}\label{thm_main}
Let $M$ be a monoid such that the $\GreenJ$-class of the identity is a
sofic group $G$, and the right translation action of $G$ on $M\setminus G$
is locally amenable. Then $M$ is sofic.
\end{theorem}

See Section~\ref{sec_remarks} below for a discussion of the extent to which
the hypotheses of Theorem~\ref{thm_main} are necessary, as well as sufficient,
conditions for soficity.

The proof is based on a fairly elementary, but quite technical, combinatorial
construction, partly inspired by the proof of \cite[Proposition~4.7]{Cecc13}:

\begin{proof}
Let $K$ be a finite subset of $M$ and $\epsilon > 0$. Let $G$ be the group
of units of $M$, and $S = M \setminus G$ the set of non-units. Note that
since $G$ is the entire $\GreenJ$-class of $1$, $S$ is an ideal of $M$. In
particular, $SG \subseteq S$, so we may consider the action by right translation
of $G$ on $S$.

For
each $k \in (K \cup K^2) \cap S$, let $H_k \mathrel{\unlhd} G$ be the pointwise stabiliser
of the orbit of
$k$, and let $G_k = G / H_k$.
Consider the induced morphism from $G$ to the direct product of the $G_k$'s.
Let $\ol{G}$ be the image of this map. For each
$g \in G$ write $\ol{g}$ for its image in $\ol{G}$.

Note that the kernel of this map
is the intersection of the subgroups $H_k$ for $k \in K \cap S$.
It follows that the right translation action of $G$ on $S$ induces a
well-defined right
action of $\ol{G}$ each orbit of an element of $(K \cup K^2) \cap S$. Indeed,
if $g,h \in G$ are such that $\ol{g} = \ol{h}$ and $s$ is in the orbit of
$k \in (K \cup K^2) \cap S$
then $g h^{-1} \in H_k$ which by definition means $sgh^{-1} = s$, so
$sg = sh$. For clarity we denote this action by $*$, so
$s * \ol{g} = sg$.

Let
$$\ol{K} \ = \ \lbrace \ol{k} \mid k \in K \cap G \rbrace \ \subseteq \ \ol{G}.$$
Choose any $\delta > 0$ sufficiently small that $(1-\delta)^3 > 1-\epsilon$.
By assumption the groups $G_k$ are all amenable; the group $\ol{G}$, being
a subgroup of a finite direct product of them, is therefore also amenable
by \cite[Proposition~4.5.1 and Corollary 4.5.6]{Cecc10}. Thus,
we may choose a finite subset $F \subseteq \ol{G}$ such that the proportion of elements
$f \in F$ satisfying $\ol{K} f \subseteq F$ exceeds $1-\delta$.

Also, recalling that $G$ is by assumption sofic, we may choose a 
finite set $P$ and a $(K \cap G, \delta)$-action of $G$ on $P$.

Now let $Z$ be a large finite set, let
$$Y \ = \ (K \cap S) \ \cup \ [(K \cap S) * F] \ \cup \ [(K^2 \cap S) * F] \ \subseteq \ S$$
and define
$$X = Y \cup (Z \times F \times P) \cup \lbrace \bot \rbrace$$
where $\bot$ is a new symbol not in any of the previous sets.

Let
$$H = \lbrace (z, f, p) \in Z \times F \times P \mid \ol{K} f \subseteq F \rbrace.$$
Clearly by choosing $Z$ large enough, we can ensure that the proportion of
elements of $X$
which come from $Z \times F \times P$ exceeds $1-\delta$, while by the definition of
$F$ the proportion of elements of $Z \times F \times P$ which lie in $H$ also exceeds
$1-\delta$. Thus, the proportion of elements of $X$ which lie in $H$
exceeds $(1-\delta)^2$.

We define a map
$$M \times X \to X, \ (m,x) \mapsto m \cdot x$$
as follows:
\begin{itemize}
\item $m \cdot n = \begin{cases} \textrm{the $M$-product } mn &\textrm{if } n \in Y \textrm{ and } mn \in Y \\
                                 \bot &\textrm{if } n \in Y \textrm{ and } mn \notin Y.
                   \end{cases}$
\item $m \cdot (z,f,p) = \begin{cases}
                           (z,\ol{m}f,m \cdot p) &\textrm{if } m \in G \textrm{ and } \ol{m} f \in F \\
                           m * f &\textrm{if } m * f \textrm{ is defined} \\
                                 \bot &\textrm{otherwise.}
                       \end{cases}$
\item $m \cdot \bot = \bot$.
\end{itemize}
(Recall that $m * f$ is defined precisely when $m$ lies in the orbit of an
element of $(K \cup K^2) \cap S$ under the right translation action of $G$.)

Our aim is to show that this map is a $(K,\epsilon)$-action of $M$ on
$X$.  We begin by recording an elementary consequence of the definition, for
ease of reference later in the proof:
\begin{itemize}
\item[(F)] If $m \in K \cap G$, then for any $(z,f,p) \in H$ we have
$\ol{m}f \in F$ (by the definition of $H$), and hence
$m \cdot (z,f,p) = (z,\ol{m}f,m \cdot p)$.
\end{itemize}

Next, note that $1 . x = x$ for all $x \in X$; indeed, we have
\begin{itemize}
\item $1 \cdot n = 1n = n$ for all $n \in Y$;
\item $1 \cdot (z,f,p) = (z, \ol{1} f, 1 \cdot p) = (z,f,p)$ for all $(z,f,p) \in Z \times F \times P$; and
\item $1 \cdot \bot = \bot$.
\end{itemize}

Now suppose $s,t \in K$. We claim first that the proportion of elements
in $(z,f,p) \in H$ which satisfy
$$s \cdot [t \cdot (z,f,p)] = (st) \cdot (z,f,p).$$
is at least $1-\delta$.
We prove the claim by analysing a number of cases, depending on whether
$s$ and $t$ are in $G$:
\begin{itemize}
\item If $s, t \in G$ then also $st \in G$. Now since we are using a $(K, \delta)$-action
of $G$ on $P$, the proportion of elements in $P$ satisfying
$s \cdot (t \cdot p) = (st) \cdot p$ is at least $1-\delta$. Since $H$
is defined as a subset of $Z \times F \times P$ by placing a restriction
only on $F$, it follows that the proportion of elements $(z,f,p) \in H$
such that $s \cdot (t \cdot p) = (st) \cdot p$ is also at least $1-\delta$.
For these elements, using fact (F) gives
$$s \cdot [t \cdot (z,f,p)] = s \cdot (z,\ol{t} f, t \cdot p).$$
Now if $(\ol{st})f = \ol{s}(\ol{t}f) \in F$ then we have 
$$s \cdot (z,\ol{t} f, t \cdot p) = (z,\ol{st} f, s \cdot (t \cdot p)) = (z,\ol{st} f, (st) \cdot p) = (st) \cdot (z,f,p),$$
while if $(\ol{st})f = \ol{s}(\ol{t}f) \notin F$ then we have 
$$s \cdot (z,\ol{t} f, t \cdot p) = \bot = (st) \cdot (z,f,p),$$
in both cases establishing the required equation.

\item If $s \notin G$ and $t \in G$ then first note that we have $st \notin G$,
since $S = M \setminus G$ is an ideal. Hence, $st \in K^2 \cap S$ which means
$(st)*f$ is defined. Also, $s \in K \cap S$, so $s*(\ol{t} f)$ is also defined.
Moreover, if we write $f = \ol{g}$ for some
$g \in G$ then we have
$$(st)*f = (st)g = s(tg) = s * \ol{tg} = s * (\ol{t}f).$$
Thus,
\begin{align*}
s \cdot [t \cdot (z,f,p)] &= s \cdot (z,\ol{t}f, t \cdot p) \ \ \ \ &\textrm{ by (F), since $t \in K \cap G, (z,f,p) \in H$} \\
                          &= s * (\ol{t}f) &\textrm{ by definition, since $s \in K \cap S$} \\
                          &= (s t) * f &\textrm{ by the preceding argument} \\
                          &= (s t) \cdot (z,f,p) &\textrm{ by definition.}
\end{align*}
so all elements of $H$ satisfy the required equation.

\item Finally, if $t \notin G$ (whether or not $s \in G$) we again have
$st \notin G$. As in the previous case, we have $(st)*f$ defined, and
this time $t*f$ is also defined. Again writing $f = \ol{g}$ for some
$g \in G$, we have
$$s \cdot (t * f) = s(t*f) = s(t \ol{g}) = (st) \ol{g} = (st) *f,$$
and hence,
\begin{align*} 
s \cdot [t \cdot (z,f,p)] &= s \cdot (t*f) \ \ \ \  &\textrm{by definition} \\
                          &= (st) * f  &\textrm{by the preceding argument} \\
                          &= (st) \cdot (z,f,p)    &\textrm{by definition}
\end{align*}
so again all elements of $H$ satisfy the required equation.
\end{itemize}
This completes the proof of the first claim.

Next let $s, t \in K$ with $s \neq t$. We claim that the proportion
of elements $(z,f,p)$ in $H$ satisfying $s \cdot (z,f,p) \neq t \cdot (z,f,p)$
is at least $1-\delta$.
Again, we consider a number of cases:
\begin{itemize}
\item  If $s,t \in G$ then, again using the $(K \cap G, \delta)$-action
of $G$ on $P$, the proportion of elements $p \in P$ such that $s \cdot p \neq t \cdot p$
is at least $1-\delta$ and it again follows that the proportion of elements
in $(z,f,p) \in H$ for which $s \cdot p \neq t \cdot p$ is at least
$1-\delta$. For such elements, by fact (F), we have
$$s \cdot (z,f,p) \ = \ (z,\ol{s}f,s \cdot p) \ \neq \ (s,\ol{t}f, t \cdot p) \ = \ t \cdot (z,f,p).$$

\item If exactly one of $s$ and $t$ lies in $G$ then using fact (F) again,
exactly one of $s \cdot (z,f,p)$ and $t \cdot (z,f,p)$ lies in $Z \times F \times P$, so
they cannot be equal for any $(z,f,p) \in H$.

\item If $s,t \notin G$ then for any $(z,f,p) \in H$, by the definition of $\cdot$ we
have we have $s \cdot (z,f,p) = s*f$ and $t \cdot (z,f,p) = t*f$. These cannot be equal,
or we would have $s = (s*f) * f^{-1} = (t*f) * f^{-1} = t$ where $f^{-1}$ is
the inverse of $f$ in the group $\ol{G}$.
\end{itemize}
This completes the proof of the second claim.

Now since the proportion of elements in $X$ which lie in $H$ exceeds
$(1-\delta)^2$, for any $s,t \in K$ the proportion of elements $x \in X$ satisfying
$s \cdot (t \cdot x) = (st) \cdot x$ is at least $(1-\delta)^3$,
which by the definition of $\delta$ is at least $1-\epsilon$. Similarly, for any $s,t \in K$ with
$s \neq t$, the proportion of elements of $X$ satisfying $s \cdot x \neq t \cdot x$
is at least $1-\epsilon$. Thus, we have defined a $(K, \epsilon)$-action of
$M$ on the finite set $X$, and so $M$ is sofic.
\end{proof}

\begin{corollary}\label{cor_conditions}
Let $M$ be a monoid such that the $\GreenJ$-class of the identity is a
sofic group $G$. If any of the following conditions hold, then $M$ is sofic:
\begin{itemize}
\item[(i)] $G$ is amenable;
\item[(ii)] $M \setminus G$ is finite; or
\item[(iii)] $M$ has finite $\GreenR$-classes outside the group of units;
\end{itemize}
\end{corollary}
\begin{proof}
\ 
\begin{itemize}
\item[(i)] If $G$ is amenable then all of its quotients are amenable
\cite[Proposition~4.5.4]{Cecc10}, so the action is locally amenable and the theorem applies.
\item[(ii)] If $M \setminus G$ is finite then orbits under the translation
action are finite, so their pointwise stabilisers must have finite index.
Thus the relevant quotients are finite, and hence by
\cite[Proposition~4.4.6]{Cecc10} amenable, so the action is locally amenable
and the theorem applies.
\item[(iii)] It is easily seen that each orbit under the right translation
action is contained in an $\GreenR$-class, so if the latter are all finite
then orbits must be finite and the same argument as in case (ii) applies.
\end{itemize}
\end{proof}

We also have the following corollary:
\begin{corollary}\label{cor_cancellative}
Every left or right cancellative monoid with amenable group of units is sofic.
\end{corollary}
\begin{proof}
Let $C$ be a left or right cancellative monoid with amenable group of units.
By the same argument
as in the proof of Proposition~\ref{prop_groupofunits} above, the $\GreenJ$-class of
the identity must be the group of units. Indeed, if it wasn't then it would contain a
copy of the bicyclic monoid, which is neither left nor right cancellative, giving a
contradiction. The result now follows from condition (i) in Corollary~\ref{cor_conditions}.
\end{proof}

It follows from a recent result of Donnelly \cite[Theorem~5]{Donnelly13}
that a cancellative left or right amenable monoid has amenable group of
units. Hence, Corollary~\ref{cor_cancellative} may be viewed
as a generalisation of \cite[Proposition 4.6]{Cecc13}, which says that
cancellative one-sided amenable monoids are sofic.

The corollaries above apply Theorem~\ref{thm_main} in rather restricted cases, although this
still suffices to give many new examples of sofic monoids. The following
example comes closer to applying Theorem~\ref{thm_main} in its full generality.

\begin{example}
Let $G$ be any sofic group, and let $S$ be any set of normal subgroups
of $G$ with amenable quotients (for example, the set of all finite index normal
subgroups). Let $M$ be the monoid generated by $S$ together with the
singleton subsets of $G$ (which we view as cosets of the trivial subgroup)
under setwise multiplication.

Then $M$ is a monoid of cosets of normal subgroups of $G$. The identity
element is the trivial subgroup $\lbrace 1 \rbrace$. The units are the
singleton subsets, which we can identify with their single elements so that
the group of units is just $G$ itself. There are
no left or right units apart the singletons, so the $\GreenJ$-class of
the identity is the group $G$.

The right translation action of the group of units is just the natural
right translation action of $G$ on the set of cosets. Each orbit
under this action is the set of all cosets of some normal subgroup $H$ of $G$.
The pointwise stabiliser of this orbit in $G$ is $H$ itself, so the
quotient by the stabiliser is $G/H$. Assuming the orbit is not the group
of units,
it follows from the definition of $M$ that $H$ is a product of subgroups
of $S$, and in particular must contain
some non-trivial subgroup $K$ from $S$. Now $G/H$ is a quotient
of $G/K$, which by
assumption is amenable, so by \cite[Proposition~4.5.4]{Cecc10}, $G/H$ is
amenable. Hence, $M$ satisfies the conditions of Theorem~\ref{thm_main}
and is sofic.
\end{example}

\section{$\GreenH$-Classes and the Action of the Group of Units}\label{sec_schutz}

In this section we consider from a structural perspective how the group of
units of a monoid acts by translation on the rest of the monoid.
Specifically, we use a standard semigroup-theoretic approach of breaking
the action down into two parts using Green's $\GreenH$-relation --- an
action \textit{on each}
$\GreenH$-class, and an action \textit{on the set of} $\GreenH$-classes
--- and study how this deconstruction relates to the local amenability
property of the action, which played such an important role in
Section~\ref{sec_soficproof} above.

Let $H$ be an $\GreenH$-class of a monoid $M$, and let $\Sigma_H$ denote
the symmetric group on the set $H$, acting on $H$ from the right and
with composition therefore from left to right. Consider the set:
$$\lbrace \sigma \in \Sigma_H \mid \textrm{ there exists } m \in M \textrm{ with } h \sigma = hm \textrm{ for all } h \in H \rbrace.$$
of all permutations of $H$ which are realised by the right translation action
of $M$ on itself. In fact this set is a subgroup \cite[Theorem~2.22]{Clifford61} of $\Sigma_H$,
called the \textit{(right) Sch\"utzenberger group} of $H$. Note that if
$H$ happens to be a subgroup (in particular, if $H$ is the group of units)
then the Sch\"utzenberger group is isomorphic to the group $H$ acting on
itself by right translation. Sch\"utzenberger
groups of $\GreenH$-classes are a powerful tool for understanding the structure
of semigroups -- see \cite[Section~2.4]{Clifford61} for a full introduction.

Now let $G$ be the group of units of the monoid $M$, and consider the
action of $G$ on $M$ by right translation. It is easily seen that if
$g \in G$ and $m \in M$ then $mg \GreenR m$; in other words, the right
translation action of $G$ preserves $\GreenR$-class. Moreover, if
$m, n \in M$ with $m \GreenL n$ then clearly $mg \GreenL ng$ so the
action preserves the $\GreenL$-relation. It follows that the action
preserves the $\GreenH$-relation, and so induces a well-defined action
on the set of $\GreenH$-classes; we denote this action by $\circ$.

\begin{theorem}\label{thm_schutz}
Let $M$ be a monoid where the $\GreenJ$-class of the identity is a sofic
group $G$, and suppose that for every non-unit $\GreenD$-class $D$ of $M$, either
or both of the following conditions hold:
\begin{itemize}
\item[(i)] there are finitely many $\GreenL$-classes in $D$, and the
Sch\"utzenberger group of $D$ is amenable; or
\item[(ii)] the Sch\"utzenberger group of $D$ is finite or abelian, and
the $\circ$-action of $G$ on the set of $\GreenH$-classes in $D$ is locally
amenable.
\end{itemize}
Then $M$ is sofic.
\end{theorem}
\begin{proof}
We shall show that the right translation action of $G$ on $M \setminus G$ is
locally amenable, so that the result follows from Theorem~\ref{thm_main}.
Let $x$ be an element of $M \setminus G$, and let $X$ be its orbit under
the right translation action of $G$. We wish to show that the quotient of
$G$ by the pointwise stabiliser of $X$ is amenable.

Let $Z$ be the union of the $\GreenH$-classes of elements in $X$.
Notice that because the right translation action of $G$ preserves
$\GreenR$-class, $X$ is contained in an $\GreenR$-class. Since
each $\GreenR$-class is a union of $\GreenH$-classes, $Z$ is also
contained in a single $\GreenR$-class, and hence also a single
$\GreenD$-class.

Now since the right translation action of $G$ preserves the
$\GreenH$-relation, $Z$ is also a union of orbits under this action.
It follows that the pointwise stabiliser of $Z$ under this action is
a normal subgroup of $G$; call it $Q$.

Now $X$ is contained in $Z$, so the pointwise stabiliser of $X$
contains that of $Z$, so the quotient of $G$ by the pointwise stabiliser
of $X$ is a quotient of $G/Q$. Since a quotient of an amenable group
is amenable \cite[Proposition~4.5.4]{Cecc10}, it will suffice to show
that $G/Q$ is amenable.

Let
$$P = \lbrace g \in G \mid h g \GreenH h \textrm{ for all } h \in Z \rbrace.$$
It is immediate from the definition that $P$ is the pointwise stabiliser of
the set of $\GreenH$-classes in $Z$, under the $\circ$-action of $G$. It
follows from the definition of $Z$ that this set is an orbit under the
$\circ$-action. Hence, $P$ is a normal subgroup of $G$.

Moreover, any element of $G$ stabilising $Z$ pointwise under the right
translation action must also stabilise the
$\GreenH$-classes in $Z$ under the $\circ$-action, which means
$Q \leq P \leq G$. It follows that
$G/Q$ is an extension of the quotient $G/P$ by the subgroup $P/Q$.
Since an extension of an
amenable group by an amenable group is amenable \cite[Proposition~4.5.5]{Cecc10},
it will thus suffice for the theorem to show that $G/P$ and $P/Q$ are both
amenable.

Consider first $G/P$. By assumption, the $\GreenD$-class containing $Z$ satisfies
either condition (i) or condition (ii) from the statement of the theorem.
In case (i), since $Z$ is contained in an
$\GreenR$-class, and each $\GreenH$-class is the intersection of an
$\GreenR$-class and an $\GreenL$-class, $Z$ can contain only finitely
many $\GreenH$-classes. It is
immediate from the definition of $P$ that $G/P$ acts faithfully on the set
of such, so it must be finite and hence amenable by \cite[Proposition~4.4.6]{Cecc10}.
 In case (ii),
the $\circ$-action is locally amenable so $G/P$, being a quotient
by the pointwise stabiliser
of an orbit, is amenable.

Next, we show that $P/Q$ is amenable. For each $\GreenH$-class $H$ in $Z$,
and each $p \in P$, let $p_H \in \Sigma_H$ denote the permutation of $H$
induced by the right translation action of $p$. Since $p \in G \subseteq M$, it is
immediate from the definition of the Sch\"utzenberger group that $p_H$ lies
in the Sch\"utzenberger group of $H$. Thus the map $p \mapsto p_H$ gives
a morphism from $P$ to the Sch\"utzenberger group. Together, these maps
induce a morphism from $P$ to the direct product of all the
Sch\"utzenberger groups of $\GreenH$-classes in $Z$. In fact, these
Sch\"utzenberger groups are all isomorphic \cite[Theorem~2.25]{Clifford61},
so it may be
viewed as a morphism from $P$ to the direct power $S^n$, where $S$ is
isomorphic to the Sch\"utzenberger groups and $n$ is the (possibly
infinite) cardinality of the set of $\GreenH$-classes in $Z$.

Notice that this map sends $p \in P$ to the identity if and only if $p$
fixes every element of every $\GreenH$-class in $Z$, that is, if $p$
fixes every element of $Z$. In other words,
the kernel of this map is exactly the subgroup $Q$. So the map induces an
embedding of the quotient $P/Q$ into $S^n$.

Again, either condition (i) or condition (ii) from the
statement of the theorem applies to the $\GreenD$-class containing $Z$.
In case (i), $n$ is finite and $S$ is amenable, and so by \cite[Corollary 4.5.6]{Cecc10},
$S^n$ is amenable.
In case (ii), the assumptions on $S$ are sufficient
to ensure that $S^n$ is amenable even if $n$ is infinite. Indeed, if $S$
is abelian then $S^n$ is abelian, and hence amenable by
\cite[Theorem~4.6.1]{Cecc10}. If $S$ is finite then $S^n$ is locally
finite, and hence amenable by \cite[Corollary 4.5.12]{Cecc10}. In all
cases, it follows that $P/Q$, being isomorphic to a subgroup of $S^n$, is
amenable by \cite[Proposition~4.5.1]{Cecc10}.
\end{proof}

A case of particular interest is where condition (i) applies to \textit{every}
$\GreenD$-class, that is, where each $\GreenD$-class contains only
finitely many $\GreenL$-classes. As well as ensuring that
the action of $G$ on the set of $\GreenH$-classes is locally amenable,
this condition also automatically rules out the presence of a bicyclic
submonoid, yielding a simpler statement:

\begin{corollary}
Let $M$ be a monoid with sofic group of units, all non-unit Sch\"utzenberger
groups amenable and finitely many $\GreenL$-classes in each $\GreenD$-class.
Then $M$ is sofic.
\end{corollary}
\begin{proof}
To apply Theorem~\ref{thm_schutz}, all we need is to show that the
$\GreenJ$-class of the identity is the group
of units. Suppose not. 
Then by the argument in the proof of Proposition~\ref{prop_groupofunits},
$M$ admits a submonoid isomorphic to the bicyclic monoid, say generated by
elements
$p$ and $q$ with $pq = 1$. Since each $\GreenD$-class has finitely many
$\GreenL$-classes and elements of the bicyclic monoid are certainly
$\GreenD$-related, by the pigeon-hole principle there must exist distinct
natural numbers $i < j$ with $p^i \GreenL p^j$. Since $\GreenL$ is
a right congruence, we get
$$1 \ = \ p^i q^i \ \GreenL \ p^j q^i \ = \ p^{j-i}.$$
Thus, there is an element $x \in M$ with $x p^{j-i} = 1$, in other words,
$p$ is a right unit. But this contradicts the fact that $p$ is not
cancellable on the left in the bicyclic monoid.
\end{proof}

Recall that a monoid $M$ is called \textit{regular} if for every
$x \in M$ there exists an element $y \in M$ with $xyx = x$. The class of regular
monoids includes in particular all inverse monoids. In a regular
monoid every Sch\"utzenberger group arises as a (maximal) subgroup
around some idempotent (this follows from \cite[Theorems~2.22 and 2.24]{Clifford61}
and the elementary fact that every $\GreenD$-class of a regular semigroup
contains an idempotent, and hence a maximal subgroup). Thus, Theorem~\ref{thm_schutz}
has the following immediate corollary in this case:

\begin{corollary}
Let $M$ be a regular monoid with sofic group of units, all non-unit
subgroups amenable and finitely many $\GreenL$-classes in each
$\GreenD$-class. Then $M$ is sofic.
\end{corollary}

\section{Remarks and Open Questions}\label{sec_remarks}

We consider the extent to which the hypotheses in Theorem~\ref{thm_main}
are necessary. The requirement that the $\GreenJ$-class of the identity
be a sofic group is necessary, by Proposition~\ref{prop_groupofunits}.
However, it is unclear to what extent the remaining hypotheses are
essential. The conditions certainly do not give an exact characterisation
of sofic monoids, since the hypotheses do not apply in general to residually
finite monoids.

\begin{example} Let $F$ be any group which is residually finite but not
amenable (for example, a free group of rank at least $2$). Let
$S = \lbrace 0,1 \rbrace$ be the \textit{$2$-element semilattice} (the
monoid with multiplication given by $1 1 = 1$ and $1 0 = 0 1 = 0 0 = 0$).
Then the direct product $F \times S$
is easily seen to be a residually finite monoid,
and hence by \cite[Corollary 4.2]{Cecc13} a sofic monoid. However, the
(non-amenable) group of units $F \times \lbrace 1 \rbrace$, which is
isomorphic to $F$, acts faithfully
and transitively on $F \times \lbrace 0 \rbrace$ (the rest of the monoid).
In other words, $F \times \lbrace 0 \rbrace$ is a single orbit with trivial
pointwise stabiliser, so the corresponding quotient is again isomorphic to
$F$ and hence non-amenable. Thus, the hypotheses of Theorem~\ref{thm_main}
are not satisfied.
\end{example}

Sofic groups themselves having been introduced as a generalisation of
amenable groups, it is natural to ask if the hypothesis of amenability
in Theorem~\ref{thm_main} can be replaced with the weaker hypothesis
of soficity.

\begin{question}\label{qn_soficsuffices1}
Let $M$ be a monoid in which the $\GreenJ$-class of the identity is a
sofic group $G$, and such that quotients of $G$ by orbit stabilisers
under the right translation action on $M$ are all sofic. Is $M$ necessarily sofic?
\end{question}

Indeed, it could even be that soficity of $M$ results from soficity
of $G$, without the necessity of considering the action on $M \setminus G$.
\begin{question}\label{qn_soficsuffices2}
Let $M$ be a monoid in which the $\GreenJ$-class of the identity is a
sofic group. Is $M$ necessarily sofic?
\end{question}
A positive answer to this question would, combined with 
Proposition~\ref{prop_groupofunits} above, completely describe
sofic monoids modulo the case of sofic groups.

A natural first step to answering either of the above questions is to 
consider the case where group of units and/or its quotients by orbit
stabilisers are residually finite but not amenable.

An interesting feature of Theorem~\ref{thm_main} (presaged by
\cite[Proposition~4.7]{Cecc13}) is that the hypotheses concerns
only the group of units and its right translation action on the
monoid: the internal multiplication of the non-unit elements plays
no role. It is natural to ask if this is only a feature of the theorem,
or if it is an inherent property of sofic monoids:
\begin{question}\label{qn_actiononly}
Let $M$ and $N$ be monoids, and assume that the $\GreenJ$-class of
the identity in each is a group. Suppose there is a bijection
$\rho : M \to N$ which restricts to an isomorphism between the groups
of units, and such that
$\rho(st) = \rho(s)\rho(t)$ whenever $t$ is a unit in $M$. If $M$ is
sofic, must $N$ also be sofic?
\end{question}

Note that a positive answer to Question~\ref{qn_soficsuffices1} would
also entail a positive answer to Question~\ref{qn_actiononly}. A
positive answer to Question~\ref{qn_soficsuffices2} would imply the
even stronger statement that soficity for monoids only ever depends
on the structure of the $\GreenJ$-class of the identity, and not even
on its action on the rest of the monoid.

Irrespective of the answer to Question~\ref{qn_actiononly}, 
\cite[Proposition~4.7]{Cecc13} (and also condition (i) of our
Corollary~\ref{cor_conditions}) implies that the soficity condition
imposes no restriction whatsoever on the
internal complexity of the monoid outside the group of units. This
suggests that soficity of monoids, as studied here, is only likely
to be of interest in applications where the group of units plays a
fundamental role.

If seeking applications in semigroup theory more widely, one is drawn to
ask if there is an alternative, probably stronger,
definition of a sofic monoid which also generalises sofic groups but
exerts more control on the internal structure of the rest of the
monoid. A natural test of whether a definition is satisfactory in
this respect would be whether the resulting class is closed under the
taking of
monoid subsemigroups (that is, subsets closed under the multiplication
and with an identity element which is not necessarily the identity
of the containing monoid). Such a definition, if found, is also likely
to extend naturally to semigroups without an identity element.

\bibliographystyle{plain}

\def\cprime{$'$} \def\cprime{$'$}

\end{document}